\documentclass[a4paper,10pt]{amsart}

\usepackage[english]{babel}
\usepackage[utf8]{inputenc}

\usepackage{mathrsfs}
\usepackage{mathtools}
\usepackage{dsfont}
\usepackage{amssymb}
\usepackage{amsmath}
\usepackage{enumerate}
\usepackage{amsfonts}
\usepackage{amsthm,amsmath}
\usepackage{bbm}
\usepackage{color}
\usepackage{hyperref}

\numberwithin{equation}{section}

\newcommand{\R}{\mathds{R}}

\newcommand{\cT}{\mathscr{T}}

\DeclarePairedDelimiter\norm{\lVert}{\rVert}

\newcommand{\argument}{\,\cdot\,}
\renewcommand{\phi}{\varphi}

\newcommand{\dx}{\:\mathrm{d}}

\definecolor{mygreen}{rgb}{0.1,0.75,0.2}

\theoremstyle{definition}
\newtheorem{definition}{Definition}

\theoremstyle{plain}

\newtheorem{theorem}[definition]{Theorem}

\begin{document}

\title[On a Convergence Theorem for Positive Semigroups]{On a Convergence Theorem for Semigroups of Positive Integral Operators}
\author{Moritz Gerlach}
\address{Moritz Gerlach\\Universit\"at Potsdam\\Institut f\"ur Mathematik\\Karl-Liebknecht-Stra{\ss}e 24--25\\14476 Potsdam\\Germany}
\email{moritz.gerlach@uni-potsdam.de}
\author{Jochen Gl\"uck}
\address{Jochen Gl\"uck\\Universit\"at Ulm\\Institut f\"ur Angewandte Analysis\\89069 Ulm, Germany}
\email{jochen.glueck@uni-ulm.de}

\date{\today}
\begin{abstract}
	We give a new and very short proof of a theorem of Greiner asserting that a positive and contractive $C_0$-semigroup on an $L^p$-space 
	is strongly convergent in case that it has a strictly positive fixed point and contains an integral operator. 
	Our proof is a streamlined version of a much more general approach to the asymptotic theory of positive semigroups developed recently by the authors. 
	Under the assumptions of Greiner's theorem, this approach becomes particularly elegant and simple. We also give an outlook on several generalisations of this result.
\end{abstract}

\maketitle

Consider a positive and contractive $C_0$-semigroup $\cT \coloneqq (T_t)_{t \in [0,\infty)}$ on $L^p \coloneqq L^p(\Omega,\mu)$, where $(\Omega,\mu)$ is a $\sigma$-finite measure space 
and $p \in [1,\infty)$. 
By positivity we mean that $f \ge 0$ implies $T_tf \ge 0$ for all $f \in L^p$ and all times $t \ge 0$. We are interested in studying the behaviour of $T_t$ as $t \to \infty$.

In applications it frequently occurs that $\cT$ consists of so-called \emph{integral operators} (or \emph{kernel operators}). 
Here, a positive linear operator $T\colon L^p \to L^p$ is called an integral operators if for a measurable function $k\colon \Omega \times \Omega \to [0,\infty)$ 
and all $f \in L^p$ the following holds: we have $k(\argument,y)f(\argument) \in L^1(\Omega,\mu)$ for almost all $y \in \Omega$ and 
$Tf = \int_\Omega k(x, \argument)f(x) \dx \mu(x)$. If at least one of the operators $T_t$ is an integral operator and if the semigroup $\cT$ has a fixed point 
which is strictly positive almost everywhere, 
then one automatically obtains strong convergence of $T_t$ as time tends to infinity. 
This was first observed by Greiner \cite[Kor~3.11]{greiner1982} and his result reads as follows.

\begin{theorem} \label{thm:greiner}
	Let $(\Omega,\mu)$ be a $\sigma$-finite measure space, 
	let $p \in [1,\infty)$ and let $\cT \coloneqq (T_t)_{t \in [0,\infty)}$ be a positive and contractive $C_0$-semigroup on $L^p \coloneqq L^p(\Omega,\mu)$. 
	If $\cT$ has a fixed point $f_0$ which fulfils $f_0 > 0$ almost everywhere and if $T_{t_0}$ is an integral operator for at least one time $t_0 \ge 0$, then $T_t$ converges strongly as $t \to \infty$.
\end{theorem}

An application of this result to semigroups generated by elliptic operators on $L^1$ can, for instance, be found in \cite{arendt2008}. 
Moreover, Theorem~\ref{thm:greiner} can be used to derive a famous result of Doob about the convergence of Markov semigroups on spaces of measures, 
see \cite{gerlach2012} and \cite[Sec~4]{gerlach2015}. Related results on $\ell^p$-sequence spaces and, more generally, on atomic measure spaces can be found in \cite{davies2005, keicher2006, wolff2008}. 

One of the major drawbacks of Theorem~\ref{thm:greiner} is its difficult proof. In fact, Greiner reduced the theorem to a $0$-$2$-law whose proof is itself technically quite involved. 
Here, we present a proof of Theorem~\ref{thm:greiner} which only uses the classical decomposition theorem by Jacobs, de Leeuw and Glicksberg 
and the observation that every positive integral operator on $L^p$ maps order intervals to relatively compact sets. 
Indeed, a rather explicit proof of the latter fact is presented in the Appendix of \cite{gerlachConvPOS}; 
however, the fact can also be deduced from abstract Banach lattice theory, see \cite[Prop~IV.9.8]{schaefer1974} and \cite[Cor~3.7.3]{meyer1991}. The methods used in the following proof are taken from a more general approach to the asymptotic theory of positive semigroup representations that was recently developed by the authors in \cite{gerlachConvPOS}. In the setting discussed here, the arguments from this approach become particularly neat and yield a surprisingly simple proof of Greiner's theorem, so we find it worthwhile to devote the present short note to this special case.

\begin{proof}[Proof of Theorem~\ref{thm:greiner}]
	We first show that the orbits of the semigroups $\cT$ are relatively compact. 
	To this end, let $c > 0$ and consider a vector $f$ in the order interval $[-cf_0,cf_0] \coloneqq \{g \in L^p: -cf_0 \le g \le cf_0\}$. For every $t \ge t_0$ we have
	\begin{align*}
		T_tf \in T_{t_0}T_{t-t_0}[-cf_0,cf_0] \subseteq T_{t_0} [-cf_0,cf_0].
	\end{align*}
	Since the latter set is relatively compact and independent of $t$, it follows that the orbit of $f$ under $\cT$ is relatively compact. 
	Since $f_0 > 0$ almost everywhere, the so-called principal ideal $\bigcup_{c>0} [-cf_0,cf_0]$ is dense in $L^p$ and as the semigroup is bounded, 
	it follows that the orbit of every vector in $L^p$ is relatively compact \cite[Lem~V.2.13]{nagel2000}.
	
	Hence, we can apply the decomposition theorem of Jacobs, de Leeuw and Glicksberg which is, for instance, described in \cite[Sec~2.4]{krengel1985}, \cite[Sec~V.2]{nagel2000} 
	and \cite[Sec~16.3]{haase2015}. 
	This theorem gives us a positive, contractive projection $P$ on $L^p$ 
	which commutes with each operator $T_t$ and which has the following properties: 
	$T_t$ converges strongly to $0$ on $\ker P$ as $t$ tends to $\infty$ and the restriction of $\cT$ to the range $F \coloneqq PE$ of $P$ 
	(which contains every fixed point of $\cT$ and which is a sublattice of $L^p$ as $P$ is contractive) 
	can be extended to a positive and contractive $C_0$-group $(S_t)_{t \in \R}$ on $F$.
	
	As $F$ is a closed sublattice of $L^p$, it is itself an $L^p$-space over some measure space. Let us show that $F$ is actually isometrically lattice isomorphic to $\ell^p(I)$ for some index set $I$. 
	It is known that this is the case if and only if every order interval in $F$ is compact; this fact follows for instance from \cite[Cor 21.13]{aliprantis1978}.
	So let $f,g \in F$ with $f \le g$. Then we have
	\begin{align*}
		[f,g]_F = T_{t_0}S_{-t_0}[f,g]_F \subseteq T_{t_0}[S_{-t_0}f,S_{-t_0}g]_E,
	\end{align*}
	where we used the subscripts $F$ and $E$ to distinguish order intervals in the spaces $F$ and $E$. The set $T_{t_0}[S_{-t_0}f,S_{-t_0}g]_E$ is relatively compact in $E$, so we conclude that $[f,g]_F$ is compact in $F$ and hence we have indeed $F \cong \ell^p(I)$. 

	Let $e_i \in \ell^p(I)$ be a canonical unit vector. Since each operator $S_t$ is an isometric lattice isomorphism, 
	$S_te_i$ is also a canonical unit vector for each $t \in \R$ and $i\in I$.
	It now follows from the strong continuity of $(S_t)_{t \in \R}$ that $S_t e_i = e_i$ for all $t$ sufficiently 
	close to $0$ and hence for all $t \in \R$ (cf.\ \cite[Prop 2.3]{wolff2008} for a more general observation).
	Thus, each operator $S_t$ is the identity map on $F$, i.e.\ each operator $T_t$ operates trivially on $F$. This proves the assertion.
\end{proof}

We point out that, combining the techniques presented here with results about positive group representations, one can derive considerable generalisations of Theorem~\ref{thm:greiner}. 
This is explained in detail in the author's recent paper~\cite{gerlachConvPOS}; we give a brief summary of those generalisations at the end of this note.

Now, we discuss a version of Theorem~\ref{thm:greiner} which does not require the semigroup to contain an integral operator 
but only to dominate a non-trivial integral operator. This result reads as follows.

\begin{theorem} \label{thm:pichor-rudnicki}
	Let $(\Omega,\mu)$ be a $\sigma$-finite measure space, let $p \in [1,\infty)$ and let $\cT \coloneqq (T_t)_{t \in [0,\infty)}$ be a positive and contractive $C_0$-semigroup 
	on $L^p \coloneqq L^p(\Omega,\mu)$. Assume that $\cT$ has a fixed point $f_0$ which fulfils $f_0 > 0$ almost everywhere and that the following assumption is fulfilled: 
	\begin{itemize}
		\item[$(*)$] For every fixed point $0 \not= f \ge 0$ of $\cT$ there exists a time $t \ge 0$ and a positive integral operator $K$ on $L^p$ such that $T_t \ge K$ and $Kf \not= 0$.
	\end{itemize}
	Then $T_t$ converges strongly as $t \to \infty$.
\end{theorem}

Note that the assumption~$(*)$ is automatically fulfilled if $\cT$ is irreducible, meaning that there exists no $\cT$-invariant band in $L^p$ except for $0$ and $L^p$, 
and if we have in addition $T_{t_0} \ge K \ge 0$ for at least one time $t_0 \ge 0$ and a non-zero integral operator $K$. 
This follows since every positive non-zero fixed point of an irreducible semigroup is strictly positive almost everywhere according to \cite[Prop~C-III-3.5(a)]{nagel1986}.

For irreducible Markov semigroups on $L^1$-spaces Theorem~\ref{thm:pichor-rudnicki} was proved by Pich{\'o}r and Rudnicki in \cite[Thm~1]{pichor2000}. 
This result has applications to various models from mathematical biology, see for instance \cite{rudnicki2002a, rudnicki2003, bobrowski2007, mackey2008, du2011, banasiak2012a, mackey2013}. 
Conditions similar to~$(*)$ also occurred in the literature on several occasions, though in a more explicit form. 
We refer for instance to \cite[pp.~308 and~309]{pichor2016} and to the introduction of the recent article~\cite{pichor2017}. 
A version of Theorem~\ref{thm:pichor-rudnicki} for irreducible semigroups on Banach lattices with order continuous norm was proved by the first author in \cite[Thm 4.2]{gerlach2013b}.

We only give a sketch of the proof of Theorem~\ref{thm:pichor-rudnicki}. 
For details we refer to \cite{gerlachConvPOS}, where the theorem is shown in a considerably more general setting, but with a more complex and technically more involved proof.

\begin{proof}[Sketch of the proof of Theorem~\ref{thm:pichor-rudnicki}]
	Since the set of all integral operators is a band within the regular operators on $L^p$, for each $t \ge 0$ we find a maximal integral operator $0 \le K_t \le T_t$
	and define $R_t \coloneqq T_t - K_t \ge 0$. As the product of a positive integral operator with a positive operator is a again an integral operator \cite[Prop~1.9(e)]{arendt1994}, 
	it easily follows from the maximality of $K_t$ and the semigroup law for $\cT$ that $R_{t+s} \le R_tR_s$ for all $s,t \ge 0$. 
	Hence, $R_{t+s}f_0 \le R_tR_sf_0 \le R_tT_sf_0 = R_tf_0$ for all $s,t \ge 0$, so $(R_tf_0)_{t \ge 0}$ decreases in norm to a vector $0 \le g \in L^p$.
	This vector fulfils $T_tg \ge R_t g = \lim_s R_tR_s f_0 \ge \lim_{s}R_{t+s}f_0 = g$ for each $t \ge 0$. 
	Since each operator $T_t$ is contractive, we conclude that actually $T_tg = R_t g = g$ and hence $K_tg = 0$ for all $t \ge 0$. 
	Our condition~$(*)$ now implies that $g = 0$, so we have shown that $R_t f_0 \searrow 0$ in norm as $t \to \infty$.
	
	Now we can see that the orbit of each vector $f \in [-f_0,f_0]$ is relatively compact. Indeed, let $\varepsilon > 0$ and choose $t_0 \ge 0$ 
	such that $\norm{R_{t_0}f_0} < \varepsilon$. For each $t \ge 0$ we obtain
	\begin{align*}
		T_{t_0+t}f \in K_{t_0}[-f_0,f_0] + [-R_{t_0}f_0, R_{t_0}f_0]
	\end{align*}
	and thus the orbit of $f$ under $\cT$ is contained in the set
	\begin{align*}
		\{T_tf: \; t \in [0,t_0]\} \cup \big( K_{t_0}[-f_0,f_0] + B_\varepsilon(0) \big),
	\end{align*}
	where $B_\varepsilon(0)$ denotes the open ball with radius $\varepsilon$ around $0$. Hence, the orbit of $f$ is totally bounded and thus relatively compact. 
	
	Since the principal ideal $\bigcup_{c > 0} [-cf_0,cf_0]$ is dense in $L^p$, we conclude that the orbit of actually every vector $f \in L^p$ under $\cT$ is relatively compact \cite[Lem~V.2.13]{nagel2000}, 
	so we can apply the Jacobs--de Leeuw--Glicksberg decomposition theorem. Now one proceeds as in the proof of Theorem~\ref{thm:greiner}. The only difficulty in this situation is to see that each order interval $[f,g]_F$ in $F$ is compact. 
	To this end, one first observes that
	\begin{align*}
		[-f_0,f_0]_F & \subseteq K_t S_{-t} [-f_0,f_0]_F +  R_tS_{-t}[-f_0,f_0]_F \\
		& \subseteq K_t[-f_0,f_0]_E + [-R_tf_0, R_tf_0]_E
	\end{align*}
	for each $t \ge 0$, where $(S_t)_{t \in \R}$ is given as in the proof of Theorem~\ref{thm:greiner}. Hence, the order interval $[-f_0,f_0]_F$ is totally bounded and thus compact.
	Now one can use that the principal ideal $\bigcup_{c > 0} c[- f_0, f_0]_F$ is dense in $F$ according to \cite[Cor~2 to Thm~II.6.3]{schaefer1974} in order to conclude that every order interval $[f,g]_F$ is compact in $F$.
\end{proof}

As mentioned above, Theorems~\ref{thm:greiner} and~\ref{thm:pichor-rudnicki} allow for considerable generalisations. First of all, Theorem~\ref{thm:greiner} remains true if $L^p$ 
is replaced with a Banach lattice $E$ with order continuous norm and if the semigroup $\cT$ is only assumed to be bounded instead of contractive. 
In this case, the proof clearly requires a bit more lattice theory.
Moreover, the range $F$ of the projection $P$ needs no longer be a sublattice but it is still a so-called lattice subspace of $E$, 
meaning that it is a lattice with respect to the order induced by $E$ but not with respect to the same lattice operations. 
We point out that even for $E = L^p$ the space $F$ is no longer an $\ell^p$-space in this case; instead, it is an atomic Banach lattice with order continuous norm. 
For more details we refer to \cite{gerlachConvPOS}.

Theorem~\ref{thm:pichor-rudnicki} can be generalised to bounded positive semigroups on Banach lattices with order continuous norm, too. 
However, the first part of the proof shows that one needs an additional technical assumption in this case: 
we have to assume that every \emph{super-fixed point} of the semigroup is a fixed point, meaning that $T_tg = g$ for all $t \ge 0$ whenever $T_tg \ge g \ge 0$ for all $t \ge 0$. 
Again, we refer to \cite{gerlachConvPOS} for details.

The most significant generalisation of Theorems~\ref{thm:greiner} and~\ref{thm:pichor-rudnicki} refers however 
to the strong continuity assumption with respect to the time parameter. 
In the proof of Theorem~\ref{thm:pichor-rudnicki}, this assumption is first employed when one uses that a set of the form $\{T_tf: \; t \in [0,t_0]\}$ is compact, but this step of the argument can easily be circumvented by using a bit more information about the Jacobs--de Leeuw--Glicksberg decomposition. Much more importantly, the proofs of Theorems~\ref{thm:greiner} and~\ref{thm:pichor-rudnicki} both use the strong continuity to deduce that the positive and contractive group $(S_t)_{t \in \R}$ acts trivially on $\ell^p(I)$. 
Yet, it turns out that this can also be deduced without strong continuity by only using algebraic properties of the additive group $\R$. 
Hence, if one is willing to invest more work in the proofs, one can show that Theorems~\ref{thm:greiner} and~\ref{thm:pichor-rudnicki} and their counterparts on Banach lattices 
remain true for semigroup representations $(T_t)_{t \in [0,\infty)}$ without any continuity or measurability assumptions with respect to $t$. 
For detailed results and proofs, we refer again to \cite{gerlachConvPOS} where it is also demonstrated that the same methods can be used 
to obtain convergence results for positive representations of more general semigroups.

\bibliographystyle{abbrv}
\bibliography{}

\end{document}